\documentclass{amsart}

\usepackage{caption}
\usepackage{subcaption}
\usepackage[dvipsnames]{xcolor}
\usepackage{tikz}
\usepackage{mathabx}
\usepackage{graphicx}
\usepackage{amsfonts}
\usepackage{amsmath}
\usepackage{amscd}
\usepackage{amssymb}
\usepackage{verbatim}
\usepackage{hyperref}
\usepackage{marginnote}

\newenvironment{proof*}{\noindent\emph{Proof}}{$\square$\smallskip}

\newtheorem{theorem}{Theorem}

\newtheorem{Definition}[theorem]{Definition}

\newtheorem{Example}[theorem]{Example}

\newtheorem{Remark}[theorem]{Remark}

\newtheorem{Exercise}[theorem]{Exercise}
\newtheorem{Exercises}[theorem]{Exercises}
\newtheorem{Notation}[theorem]{Notation}
\newtheorem{Convention}[theorem]{Convention}
\newtheorem{standing assumption}[theorem]{Standing Assumption}

\title[The planar pure braid group is a diagram group]{The planar pure braid group is a diagram group} 
\author[D.~S.~Farley]{Daniel S. Farley}
\address{Department of Mathematics\\ Miami University\\ Oxford, OH 45056 U.S.A.}
\email{farleyds@muohio.edu}

\date{\today}

\begin{document}

\begin{abstract} 
The \emph{$n$-strand planar pure braid group} is the fundamental group of $\mathbb{R}^{n} - \Delta$, where $\Delta$ consists of all $n$-tuples in which three or more coordinates are the same.   

We prove that the $n$-strand planar pure braid group is a diagram group, for all $n \geq 1$. Here, ``diagram group" refers to the class of groups studied at length by Guba and Sapir. We can deduce numerous corollaries. For instance, the $n$-strand planar pure braid group acts freely and cocompactly on a CAT(0) cubical complex. It is linear, bi-orderable, and biautomatic. The representation of the planar pure braid groups as diagram groups also offers procedures for computing group presentations and homology.  

We similarly find that ``cylindrical" versions of the planar pure braid group are annular diagram groups. This establishes some of the above corollaries for the cylindrical (or annular) planar pure braid groups. 

After initially posting a version of this preprint to the arXiv, the author learned that the main theorem (Theorem \ref{theorem:main}) had previously been observed by Anthony Genevois. We refer the reader to the acknowledgments for details.
\end{abstract}

\keywords{diagram groups, CAT(0) cubical complex, planar braid groups}
\subjclass[2020]{Primary 20F36, 20F65; Secondary 57M07}

\maketitle

\section{Introduction}

Consider the space $X_{n} = \mathbb{R}^{n} - \Delta$, where $\Delta$ is the set of all $n$-tuples of real numbers $(x_{1}, \ldots, x_{n})$, such that $x_{i} = x_{j} = x_{k}$, for some $1 \leq i < j < k \leq n$. The fundamental group $\Gamma_{n}$ of this space has been studied under several different names (for a short guide, see the introduction to \cite{MR}), and in mathematical contexts ranging from low-dimensional topology \cite{Khovanov2} to topological robotics \cite{Gonz}. We will follow \cite{MR} in calling $\Gamma_{n}$ the \emph{$n$-strand planar pure braid group}.


In this short note, we will show that $\Gamma_{n}$ is a diagram group, for all $n \geq 1$. Indeed, it is not difficult to observe, after comparing definitions, that $\Gamma_{n} \cong \mathcal{D}(\mathcal{P}_{n},w_{n})$, where $w_{n} = x_{1}x_{2}\ldots x_{n}$ and 
\[ \mathcal{P}_{n} = \langle x_{1}, x_{2}, \ldots, x_{n} \mid  x_{i}x_{j} = x_{j}x_{i} \, \, (i \neq j) \rangle. \]
This is our main theorem.
 We will briefly review the necessary definitions in the body of the paper.
 
 The author proved that each diagram group acts freely and cellularly on an explicitly defined CAT(0) cubical complex \cite{Farley1}. The quotient of the cube complex by this action is called the \emph{Squier complex} in \cite{G}. (Note that, in \cite{GS1}, it is the $2$-skeleton of the above quotient that is called the Squier complex. In this note, we use the definition from \cite{G}.)  Theorem 3.13 from \cite{Farley1} easily implies that the Squier complex is compact in the case of $\mathcal{D}(\mathcal{P}_{n},w_{n})$. 
 
Various corollaries now follow from the work of others. 
 Genevois \cite{G} proves that any compact Squier complex is special. In particular, the associated diagram group embeds into a right-angled Artin group by results of \cite{HW}, and is therefore linear \cite{HsuWise, Humphries} and bi-orderable \cite{DT}. A result of Niblo and Reeves \cite{NR1} shows that any group acting properly and cocompactly on a CAT(0) cubical complex is biautomatic. It follows directly that each $\Gamma_{n}$ is linear, bi-orderable, and biautomatic. Guba and Sapir have proved that all diagram groups are bi-orderable \cite{GS3}, which gives a second proof that the $\Gamma_{n}$ are bi-orderable. They have also computed presentations of diagram groups and the integral homology groups of diagram groups (see \cite{GS1} and \cite{GS2}, respectively).
 
 Several of the above corollaries also follow from the literature. The introduction of \cite{MR} clearly shows that, for each $n$, $\Gamma_{n}$ is a finite-index subgroup of a right-angled Coxeter group. Coxeter groups are well-known to be linear, so the linearity of $\Gamma_{n}$ follows directly. Moreover, the main result of \cite{NR2} implies that right-angled Coxeter groups act properly and cocompactly on CAT(0) cubical complexes, which shows that $\Gamma_{n}$ acts properly and cocompactly on a CAT(0) cubical complex. The latter fact implies that $\Gamma_{n}$ is biautomatic \cite{NR1}. We note also that presentations of $\Gamma_{n}$ \cite{M, MR} and integral homology groups of $\Gamma_{n}$ \cite{BW} have been computed. 
 
Thus, after comparing the above lists of corollaries, it appears that the proof that $\Gamma_{n}$ is bi-orderable is our main new result. Nevertheless, the literature on diagram groups offers a number of additional techniques that may prove useful in the further study of planar pure braid groups.

\emph{Acknowledgments}: I first learned about the planar braid groups during a stimulating talk by Jes\'{u}s Gonz\'{a}lez at the ``2nd International Workshop on Topology and Geometry" at the National University of San Marcos, Lima, Peru, in 2019. It is a pleasure to thank him, and the organizers of the conference. I would also like to thank Gabriel Drummond-Cole for encouraging me to write a short note on this topic.

[Added (9/8/2021): After I posted this preprint to the arXiv, Anthony Genevois informed me that Theorem \ref{theorem:main} occurs as part of Example 5.43 from \cite{G2} (or Example 5.40 from the arXiv version). His example also showed that the planar pure braid group is acylindrically hyperbolic. I would also like to thank Shane O Rourke for pointing out a correction.]
 
 \section{Diagram Groups}
 
 In this section, we offer a short definition of semigroup pictures and diagram groups. The reader who would like a more thorough introduction should refer to \cite{Farley2} or (especially) \cite{GS1}.
 
A \emph{semigroup presentation} $\mathcal{P}$ is a pair $\langle \Sigma \mid \mathcal{R} \rangle$, where $\Sigma$ is a finite set (or \emph{alphabet}) and $\mathcal{R}$ is a collection of equalities between positive, non-empty words in $\Sigma$. We consider three types of ingredients: 
\begin{itemize}
\item a non-empty collection of \emph{wires}, each of which is homeomorphic to $[0,1]$;
\item a possibly-empty collection of \emph{transistors}, each of which is a homeomorph of $[0,1]^{2}$. A non-empty collection of points on the top and bottom of any given transistor are \emph{contacts}; each contact is labeled by a letter from $\Sigma$. (We do not allow the corners of a transistor to be contacts.) If we read the labels of the contacts on a given transistor from left to right, we obtain non-empty positive words $w_{t}$ and $w_{b}$ in the generators (from the top and bottom of the transistor, respectively). We then say that the transistor in question is a \emph{$(w_{t},w_{b})$-transistor}. The equality $w_{t} = w_{b}$ (or $w_{b}=w_{t}$) is required to be among the equalities of $\mathcal{R}$ (note: and \emph{not} merely a consequence of the latter);
\item a \emph{frame}, which is a homeomorph of $\partial [0,1]^{2}$. Certain points on the top and bottom of the frame (but never the corners) are also called contacts and are given labels from the alphabet $\Sigma$. By reading the labels on these contacts from left to right, we obtain non-empty positive words $u$ and $v$ (from the top and bottom, respectively).  
\end{itemize}
We assemble the above in $\mathbb{R}^{2}$ subject to certain constraints. The frame is to enclose all of the wires and transistors. Each wire is to be monotonic (i.e., no horizontal line can cross the same wire twice), and no two wires or two transistors may touch. Each wire should connect a pair of contacts having the same label, and every contact is to be the endpoint of some (unique) wire. 
The resulting figure is a \emph{$(u,v)$-semigroup picture over $\mathcal{P}$}. Two such semigroup pictures are equivalent if they are isotopic by a label-matching isotopy. 

Given a $(u,v)$-semigroup picture $\Delta_{1}$ over $\mathcal{P}$ and a $(v,w)$-semigroup picture $\Delta_{2}$ over $\mathcal{P}$, we define the product $\Delta_{1} \circ \Delta_{2}$ by stacking $\Delta_{1}$ on top of $\Delta_{2}$. This product is well-defined up to isotopy and $\Delta_{1} \circ \Delta_{2}$ is a $(u,w)$-semigroup picture over $\mathcal{P}$. Additionally, it is not difficult to see that $\circ$ is associative when the relevant products exist. If we fix a positive non-empty word $u$ in the alphabet $\Sigma$, then the collection of all $(u,u)$-semigroup pictures over $\mathcal{P}$ becomes a semigroup. 
  
\begin{figure}
\centering
\begin{minipage}{.5\textwidth}
  \centering
\begin{tikzpicture}
\filldraw[black, thick]
 (.6, .96) rectangle (1,1.2); 
\filldraw[black, thick] 
(1.4,1.88) rectangle (1.8,2.12); 
\draw[black, thick]
(.4,2.8) .. controls (.4,2.4) and (.7,1.6) .. (.7,1.2) ; 
\draw[black, thick] 
(1.2,2.8) .. controls (1.2,2.4) and (1.5,2.52)  .. (1.5,2.12); 
\draw[black, thick]
(2,2.8) .. controls (2,2.4) and (1.7,2.52) .. (1.7,2.12); 
\draw[black, thick]
(1.5,1.88) .. controls (1.5,1.72) and (.9,1.36)    .. (.9,1.2); 
\draw[black, thick]
(.7,.96) .. controls (.7, .56) and (.4,.4) .. (.4,0) ; 
\draw[black, thick]
(.9,.96) .. controls (.9, .56) and (1.2,.4) .. (1.2,0); 
\draw[black, thick]
(1.7,1.88) .. controls (1.7,1.48) and (2,.4) .. (2,0); 
\draw[black, thick, dashed] (0,0) rectangle (2.4,2.8); 
\node at (.4,3) {$\mathbf{x_{1}}$}; 
\node at (1.2,3) {$\mathbf{x_{2}}$}; 
\node at (2,3) {$\mathbf{x_{3}}$}; 
\node at (.4, -.25) {$\mathbf{x_{3}}$}; 
\node at (1.2, -.25) {$\mathbf{x_{1}}$}; 
\node at (2,-.25) {$\mathbf{x_{2}}$}; 
\node at (1.45,1.3) {$\mathbf{x_{3}}$}; 
\end{tikzpicture}
  
\end{minipage}%
\begin{minipage}{.5\textwidth}
  \centering
\begin{tikzpicture}
\draw[black, thick, dashed] (4.4,1.4) ellipse (.7 and 1.4); 
\draw[black, thick] (4.75,2.612) .. controls (4.75, 2.312) and (4.5,2.32) .. (4.5,2.02); 
\draw[black, thick] (4.05,2.612) .. controls (4.05, 2.312) and (4.3,2.32) .. (4.3,2.02); 
\node at (3.85,2.712) {$\mathbf{x_{i}}$}; 
\node at (4.95,2.712) {$\mathbf{x_{j}}$}; 
\node at (3.85,.16) {$\mathbf{x_{i}}$}; 
\node at (5.075,.16) {$\mathbf{x_{j}}$}; 
\node at (6.65,2.712) {$\mathbf{x_{i}}$}; 
\node at (7.75,2.712) {$\mathbf{x_{j}}$}; 
\node at (6.65,.16) {$\mathbf{x_{i}}$}; 
\node at (7.875,.16) {$\mathbf{x_{j}}$}; 
\filldraw[black, thick] (4.2,.78) rectangle (4.6, 1.02); 
\filldraw[black, thick] (4.2,1.78) rectangle (4.6, 2.02); 
\draw[black, thick] (4.75,.188) .. controls (4.75,.488) and (4.5,.48) .. (4.5,.78); 
\draw[black, thick] (4.05,.188) .. controls (4.05,.488) and (4.3,.48) .. (4.3,.78); 
\draw[black, thick] (4.3,1.78) -- (4.3,1.02); 
\draw[black, thick] (4.5,1.78) -- (4.5,1.02); 
\draw[black, thick, ->] (5.4,1.4) -- (6.2,1.4); 
\draw[black, thick, dashed] (7.2,1.4) ellipse (.7 and 1.4); 
\draw[black, thick] (6.85,2.612) .. controls (7.1,1.4) .. (6.85,.188); 
\draw[black, thick] (7.55,2.612) .. controls (7.3,1.4) .. (7.55,.188); 
\end{tikzpicture}
\end{minipage}
\caption{On the left, we have an $(x_{1}x_{2}x_{3}, x_{3}x_{1}x_{2})$-picture over the presentation $\mathcal{P}_{3} = \langle x_{1}, x_{2}, x_{3} \mid 
x_{i}x_{j} = x_{j}x_{i} \, (i \neq j)\rangle$. The right half of the figure depicts the operation of reducing a dipole in a diagram over $\mathcal{P}_{n}$.}
\label{figure:picture}
\end{figure}
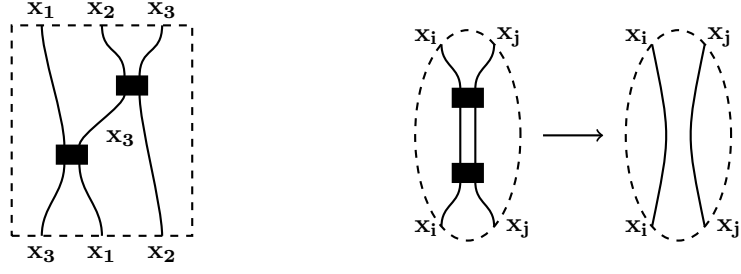

If a $(u,v)$-transistor $T_{1}$ occurs above a $(v,u)$-transistor $T_{2}$ in such a way that the wires issuing from the bottom contacts of $T_{1}$ connect bijectively and in an order-preserving fashion to the top contacts of $T_{2}$, then we say that $T_{1}$ and $T_{2}$ form a \emph{dipole}. We \emph{delete} the dipole by removing the transistors $T_{1}$ and $T_{2}$ and the wires connecting them, and then gluing the wires that were attached to the top of $T_{1}$, in order, to the wires that  were attached to the bottom of $T_{2}$. The inverse of the latter operation is called \emph{inserting} a dipole. We say that two semigroup pictures over $\mathcal{P}$ are \emph{equivalent modulo dipoles} if one can be obtained from the other by repeatedly inserting and/or deleting dipoles. ``Equivalent modulo dipoles" is easily seen to be an equivalence relation on semigroup diagrams over $\mathcal{P}$. 

The set of all $(u,u)$-semigroup diagrams over $\mathcal{P}$ modulo dipoles make up a group under the operation $\circ$, denoted $\mathcal{D}(\mathcal{P},u)$ and called the \emph{diagram group over $\mathcal{P}$ at the base word $u$}. 

\section{Planar Pure Braids}

We will use the description of planar braids that appears in \cite{MR}. 
Fix a natural number $n$. A \emph{planar braid} is a collection of $n$ smooth descending arcs in $\mathbb{R}^{2}$ connecting $n$ distinct points on one horizontal line with $n$ distinct points directly beneath them on another horizontal line. A given point in $\mathbb{R}^{2}$ is allowed to be on two of the arcs, but never on three.
 
Two planar braids $b_{1}$ and $b_{2}$ are \emph{isotopic} if there is an endpoint-preserving isotopy of the plane carrying $b_{1}$ onto $b_{2}$ such that all intermediate homeomorphisms in the isotopy also carry $b_{1}$ to planar braids. We then consider the isotopy classes of braids modulo the additional relation that allows adjacent strands of the planar braid to pass over each other, as pictured in Figure \ref{figure:ppb}. The elements of the \emph{$n$-strand planar braid group}, denoted $\overline{B}_{n}$,
are the resulting equivalence classes. The product $b_{1}b_{2}$ is defined by stacking $b_{1}$ on top of $b_{2}$.

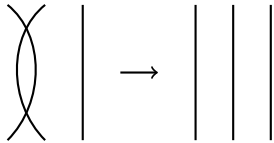
\begin{figure}[!t]
\begin{tikzpicture}
\draw[black, thick] (.5,0) .. controls (1,.5) and (1,1.4) .. (.5,1.8); 
\draw[black, thick] (1,0) .. controls (.5,.5) and (.5,1.4) .. (1,1.8); 
\draw[black, thick] (1.5,0) -- (1.5,1.8); 
\draw[black, thick, ->] (2,.9) -- (2.5,.9); 
\draw[black, thick] (3,0) -- (3,1.8); 
\draw[black, thick] (3.5,0) -- (3.5,1.8); 
\draw[black, thick] (4,0) -- (4,1.8); 
\end{tikzpicture}
\caption{We have depicted an instance of the planar braid relation, in the case of a three-strand braid.}
\label{figure:ppb}
\end{figure}

The \emph{$n$-strand planar pure braid group} $\overline{P}_{n}$ is the subgroup of $\overline{B}_{n}$ consisting of planar braids in which the initial and terminal points of each descending arc lie on a vertical line.  

\section{Proof of the main theorem}

\begin{theorem} \label{theorem:main}
Let $\mathcal{P}_{n} = \langle x_{1}, \ldots, x_{n} \mid x_{i}x_{j} = x_{j}x_{i} \, (i \neq j) \rangle$. 
The diagram group $\mathcal{D}(\mathcal{P}_{n}, x_{1}\ldots x_{n})$ is isomorphic to the $n$-strand planar pure braid group $\overline{P}_{n}$.
\end{theorem}

\begin{proof}
We define a map $\phi: \overline{P}_{n} \rightarrow \mathcal{D}(\mathcal{P}_{n},x_{1}\ldots x_{n})$ as follows. Given a planar pure braid $b$, we first enclose the braid in a frame in such a way that each strand connects a point on the top of the frame to a point on the bottom of the frame. The points of intersection at the top of the frame are labelled, from left to right, $x_{1}$ through $x_{n}$. The points of intersection at the bottom of the frame are labelled in the same way. We can now speak of the ``$x_{i}$-strand", for $i=1, \ldots, n$, which is the unique strand connecting the point labelled $x_{i}$ at the top, to the point labelled $x_{i}$ at the bottom. These labels are necessarily the same, since $b$ is a planar pure braid. 

We can assume, up to equivalence, that each crossing in $b$ is transverse. Let $c$ be any such crossing. Assume that the $x_{i}$-strand crosses over the $x_{j}$-strand from left to right as we move down the braid diagram. In this case, we replace $c$ with an $(x_{i}x_{j},x_{j}x_{i})$-transistor. Repeating this procedure at each crossing, we eventually arrive at an $(x_{1}\ldots x_{n}, x_{1}\ldots x_{n})$-picture over the presentation $\mathcal{P}_{n}$. Conversely, taking any $(x_{1}\ldots x_{n},x_{1}\ldots x_{n})$-picture $\Delta$, we can remove all labels and replace each transistor with a transverse crossing. The result is a planar pure braid, which maps to $\Delta$ under $\phi$. 

Note that the planar braid equivalence depicted in Figure \ref{figure:ppb} has its exact counterpart in the procedure of reducing dipoles (see the right half of Figure \ref{figure:picture}). It follows that the correspondence $\phi$ is a well-defined bijection. Finally, we note that the operation (of stacking) is identical in both groups, which makes $\phi$ an isomorphism.
\end{proof}

\begin{theorem} \label{theorem:annular} Let $\mathcal{P}_{n}$ be as defined in Theorem \ref{theorem:main}.
The annular diagram group $D_{a}(\mathcal{P}_{n}, x_{1}\ldots x_{n})$ is isomorphic to the annular planar pure braid group.
\end{theorem}

\begin{proof}
The proof is similar to that of Theorem \ref{theorem:main}. The reader who is interested in the precise definition of annular diagram groups can consult \cite{Farley2} and \cite{GS1}. We indicate the isomorphism in Figure \ref{figure:three}.
\end{proof}

 \begin{figure}[!t]
\centering
\begin{minipage}{.5\textwidth}
\centering
\begin{tikzpicture}
\draw[black,thick,dashed] (1.4,1.4) circle (1.4); 
\draw[black,thick,dashed] (1.4,1.4) circle (.5); 
\filldraw[black] (.9,1.4) circle (1.5pt); 
\filldraw[black] (0,1.4) circle (1.5pt); 
\filldraw[black] (1.047,1.047) circle (1pt); 
\filldraw[black] (1.753,1.047) circle (1pt); 
\filldraw[black] (1.4,.9) circle (1pt); 
\filldraw[black] (.411,.411) circle (1pt); 
\filldraw[black] (1.4,0) circle (1pt); 
\filldraw[black] (2.389,.411) circle (1pt); 
\filldraw[black] (.806,.948) -- (.948,.806) -- (.863,.721) -- (.721,.863) -- cycle; 
\filldraw[black] (1.3,.36) -- (1.5,.36) -- (1.5,.24) -- (1.3,.24) -- cycle; 
\draw[black, thick] (1.047,1.047) -- (.913,.842) [out=225,in=45]; 
\draw[black, thick] (.757,.828) -- (.411,.411) [out=225,in=45]; 
\draw[black, thick] (1.753,1.047) to [out=315,in=180] (1.95,.93) to [out=0,in=270] (2.05,1.4); 
\draw[black, thick] (2.05,1.4) arc (0:180:.65); 
\draw[black, thick] (.75,1.4) .. controls (.86,.95) .. (.842,.913) [out=270,in=45]; 
\draw[black,thick] (1.4,.9) -- (1.45,.36) [out=270,in=90]; 
\draw[black,thick] (.828,.757) .. controls (.7,.6) and (1.35,.5) .. (1.35,.36) [out=225,in=90]; 
\draw[black,thick] (1.35,.24) -- (1.4,0) [out=270,in=90]; 
\draw[black,thick] (1.45,.24) .. controls (1.6,.01) and (2,.6) .. (2.389,.411) [out=270,in=135]; 
\end{tikzpicture}
\end{minipage}%
\begin{minipage}{.5\textwidth}
\centering
\begin{tikzpicture}
\draw[black,thick,dotted] (1.4,1.4) circle (1.4); 
\draw[black,thick,dotted] (1.4,1.4) circle (.5); 
\filldraw[black] (.9,1.4) circle (1pt); 
\filldraw[black] (0,1.4) circle (1pt); 
\filldraw[black] (1.047,1.047) circle (1pt); 
\filldraw[black] (1.753,1.047) circle (1pt); 
\filldraw[black] (1.4,.9) circle (1pt); 
\filldraw[black] (.411,.411) circle (1pt); 
\filldraw[black] (1.4,0) circle (1pt); 
\filldraw[black] (2.389,.411) circle (1pt); 
\draw[black,thick] (1.047,1.047) -- (.411,.411); 
\draw[black,thick] (1.4,.9) -- (1.4,0); 
\draw[black, thick] (1.753,1.047) to [out=315,in=180] (1.95,.93) to [out=0,in=270] (2.1,1.4); 
\draw[black, thick] (2.1,1.4) arc (0:270:.8); 
\draw[black,thick] (1.3,.6) .. controls (1.4,.6) and (2,.6) .. (2.389,.411) [out=0,in=135]; 
\end{tikzpicture}
\end{minipage}
\caption{On the left, we have an annular $(x_{1}x_{2}x_{3},x_{1}x_{2}x_{3})$-picture over $\mathcal{P}_{3}$ and, on the right, its corresponding annular planar pure braid. The dots on the left sides of the inner and outer circles are basepoints that are used to define the operation.}
\label{figure:three}

\end{figure}
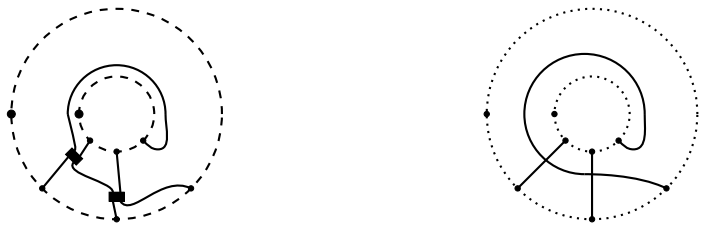

 \bibliographystyle{plain}
\bibliography{biblio}

\end{document}